\documentclass[11pt]{amsart}
\usepackage{amsmath,amsthm,amsfonts,amssymb,graphicx,psfrag,dsfont,mathrsfs,amscd}
\usepackage{color}

\usepackage{hyperref}
\usepackage[initials]{amsrefs}
\usepackage[all,cmtip]{xy}

\theoremstyle{plain}
\newtheorem{thm}{Theorem}[section]
\newtheorem{prop}[thm]{Proposition}
\newtheorem{conj}[thm]{Conjecture}

\newtheorem{lemma}[thm]{Lemma}
\newtheorem{question}{Question}

\theoremstyle{definition}
\newtheorem{example}[thm]{Example}

\newtheorem{defn}[thm]{Definition}

\theoremstyle{remark}
\newtheorem{remark}[thm]{Remark}

\renewcommand{\SS}{{\mathbb S}}

\newcommand{\CC}{{\mathbb C}}

\newcommand{\RR}{{\mathbb R}}

\newcommand{\NN}{{\mathbb N}}

\newcommand{\ZZ}{{\mathbb Z}}

\newcommand{\nN}{{\mathcal N}}
\newcommand{\aA}{{\mathcal A}}
\newcommand{\cC}{{\mathcal C}}

\newcommand{\mM}{{\mathcal M}}

\newcommand{\op}{\operatorname}

\newcommand{\Spinc}{\op{Spin}^c}

\newcommand{\End}{\op{End}}

\renewcommand{\ker}{\op{Ker}}

\newcommand{\cl}{\op{cl}}
\newcommand{\PD}{\op{PD}}

\renewcommand{\1}{\mathds{1}}
\newcommand{\e}{\varepsilon}

\newcommand{\bA}{\mathbf A}

\newcommand{\fs}{\mathfrak s}

\newcommand{\fd}{\mathfrak d}
\newcommand{\fM}{\mathfrak M}

\newcommand{\fc}{\mathfrak c}

\newcommand{\Hfrom}{\widehat{\mathit{HM}}}
\newcommand{\Cfrom}{\widehat{\mathit{CM}}}

    \RequirePackage{rotating}                   
    \def\Hto{%
       \setbox0=\hbox{$\widehat{\mathit{HM}}$}
       \setbox1=\hbox{$\mathit{HM}$}
       \dimen0=1.1\ht0
       \advance\dimen0 by 1.17\ht1
       \smash{\mskip2mu\raise\dimen0\rlap{%
          \begin{turn}{180}
              {$\widehat{\phantom{\mathit{HM}}}$}
           \end{turn}} \mskip-2mu    
                \mathit{HM}
    }{\vphantom{\widehat{\mathit{HM}}}}{}}
    \RequirePackage{rotating}                   
    \def\Cto{%
       \setbox0=\hbox{$\widehat{\mathit{CM}}$}
       \setbox1=\hbox{$\mathit{CM}$}
       \dimen0=1.1\ht0
       \advance\dimen0 by 1.17\ht1
       \smash{\mskip2mu\raise\dimen0\rlap{%
          \begin{turn}{180}
              {$\widehat{\phantom{\mathit{CM}}}$}
           \end{turn}} \mskip-2mu    
                \mathit{CM}
    }{\vphantom{\widehat{\mathit{CM}}}}{}}
    
\numberwithin{equation}{section}

\definecolor{blue}{rgb}{0,0,1}
\definecolor{red}{rgb}{1,0,0}
\definecolor{green}{rgb}{0,.7,0}

\title[No homotopy 4-sphere invariants using ECH=SWF]{No homotopy 4-sphere invariants using $ECH = SWF$}

\author{Chris Gerig}
\date{\today}
\address{Department of Mathematics\\
Harvard University\\
MA 02138\\
USA}
\email{cgerig@math.harvard.edu}

\begin{document}

\begin{abstract}
In relation to the 4-dimensional smooth Poincar\'e conjecture we construct a tentative invariant of homotopy 4-spheres using embedded contact homology (ECH) and Seiberg--Witten theory (SWF). But for good reason it is a constant value independent of the sphere, so this null-result demonstrates that one should not try to use the usual theories of ECH and SWF. On the other hand, a corollary is that there always exist pseudoholomorphic curves satisfying certain constraints in (punctured) 4-spheres.
\end{abstract}

\maketitle

\section{(Near-)symplectic geometry}

Let $X$ be a homotopy 4-sphere.\footnote{That is, $\pi_\ast(X)\cong\pi_\ast(S^4)$, hence $H_\ast(X;\ZZ)\cong H_\ast(S^4;\ZZ)$ by the Hurewicz theorem.} In particular, it is a closed simply-connected orientable smooth 4-manifold with $b^2_+(X)=0$ so that the Seiberg--Witten invariants do not apply. As a topological manifold it is homeomorphic to $S^4$ by the solved Poincar\'e conjecture, thanks to Freedman \cite{Freedman:Poincare}. Whether $X$ has an exotic smooth structure is the result of

\begin{conj}[4-dimensional smooth Poincar\'e conjecture]
\label{conj}
$X$ is diffeomorphic to $S^4$.
\end{conj}

Let $X^\ast$ be the noncompact manifold obtained by puncturing $X$, or equivalently by removing a small standard (closed) 4-ball from $X$ and smoothly attaching an end of the form $[0,\infty)\times S^3$. We say that $X^\ast$ is \textit{asymptotically Euclidean}, i.e. the complement of some compact set is diffeomorphic to the complement in the standard $\RR^4$ of some compact set. To discuss the relation between smooth structures on $X$ and $X^\ast$ we recall Cerf's ``$\Gamma_4=0$'' theorem and Cerf--Palais' ``disk'' theorem, which together imply that there is a unique way to remove or replace a standard 4-ball.

\begin{thm}[Cerf \cite{Cerf:theorem, Eliashberg:20, GeigesZehmisch:Cerf}]
All orientation-preserving diffeomorphisms of the sphere $S^3$ extend to a diffeomorphism of the standard ball $D^4$.
\end{thm}

\begin{thm}[Cerf--Palais \cite{Cerf:topologie, Palais:natural}]
\label{CerfPalais}
Any two equi-oriented embeddings of a closed 4-ball into a connected 4-manifold are ambient isotopic.
\end{thm}

Cerf's theorem implies that all orientation-preserving diffeomorphisms of $S^3$ are smoothly isotopic to the identity, and in particular, there are no ``twisted'' exotic 4-spheres (i.e. those which are obtained from two standard 4-balls by gluing their boundary via a diffeomorphism of $S^3$). Cerf--Palais' theorem implies that $X^\ast$ is diffeomorphic to the standard $\RR^4$ if $X$ is diffeomorphic to the standard $S^4$, by isotoping the embedding of a standard 4-ball to the northern hemisphere of $S^4$. But the converse is unknown, meaning it might be possible that $X$ is exotic and $X^\ast$ is standard, in which case if we view $X^\ast$ as an open 4-ball then its closure is an exotic closed 4-ball (with standard interior).\footnote{Private communication with Gompf. The statement ``$X$ exotic $\Leftrightarrow$ $X^\ast$ exotic'' holds true if the 3-dimensional smooth Schoenflies conjecture is answered in the affirmative, see Gompf \cite{Gompf:groupExotic}*{Proposition 2.2}.} Note that if we glue a standard 4-ball to an exotic closed 4-ball along the common $S^3$ boundary then the resulting 4-sphere is exotic, otherwise we could remove a standard open 4-ball from $S^4$ to get a contradiction using Theorem~\ref{CerfPalais}.

\bigskip
The reason we pass from $X$ to $X^\ast$ is that $X$ does not admit symplectic 2-forms nor near-symplectic 2-forms but $X^\ast$ admits both, thanks to Gromov's h-principle theorem on open manifolds (Eliashberg \& Mishachev \cite{EliashbergMishachev}*{\S10.2.2}). Moreover, we will see momentarily that there are suitably nice near-symplectic forms to equip $X^\ast$ with, and that there is a crucial characterization of the symplectic forms. This characterization is known as Gromov's ``recognition of $\RR^4$'', stated below. It implies that there are no exotic symplectic structures on $\RR^4$ that are \textit{asymptotically standard} (up to compactly supported symplectomorphisms), i.e. such that they agree with the standard symplectic form on $\RR^4$ outside a compact set.  For the record, by ``standard symplectic $\RR^4$'' we mean the total space of $T^\ast\RR^2$ with canonical coordinates $(x_1,x_2,y_1,y_2)$ and symplectic form $\omega_{\op{std}}=dx_1\wedge dy_1+dx_2\wedge dy_2$.

\begin{thm}[Gromov \cite{Gromov}*{\S0.3.C}]
Let $(M,\omega)$ be a noncompact symplectic 4-manifold whose reduced integral homology is trivial. Suppose there are compact sets $K_M\subset M$ and $K_{\RR^4}\subset\RR^4$ with a symplectomorphism $\phi:(M-K_M,\omega)\to(\RR^4-K_{\RR^4},\omega_{\op{std}})$.  Then $\phi$ extends to a symplectomorphism $(M,\omega)\to (\RR^4,\omega_{\op{std}})$ after removing slightly bigger compact sets.
\end{thm}

\begin{remark}
Gromov also proved that the group of compactly supported symplectomorphisms of the standard symplectic $(\RR^4,\omega_{\op{std}})$ is contractible, so the space of asymptotically standard symplectic forms on $\RR^4$ is homotopy equivalent to the group of all compactly supported diffeomorphisms. Eliashberg asked about the topology of this group in \cite{Eliashberg:90}*{\S7}. Related results are given by Eliashberg and McDuff in \cite{Eliashberg:properties, McDuff:boundaries}.
\end{remark}

We now remind the reader of the definition of a near-symplectic form. Using $M$ to temporarily denote either $X^\ast$ or any closed 4-manifold, a closed 2-form $\omega:M\to\bigwedge^2T^\ast M$ is \textit{near-symplectic} if for all points $x\in M$ either $\omega^2(x)>0$, or $\omega(x)=0$ and the rank of the gradient $\nabla\omega_x:T_xM\to\bigwedge^2T_x^\ast M$ is three. In other words, $\omega$ is symplectic on $M-\omega^{-1}(0)$ and vanishes transversely on its zero set $\omega^{-1}(0)$ which consists of a finite disjoint union of smooth embedded circles. The following result was asserted by Taubes in \cite{Taubes:Luttinger}, and we sketch the proof in the Appendix~\ref{appendix}.

\begin{thm}
\label{LuttingerTaubes}
There exist (exact) near-symplectic forms on $X^\ast$ which are asymptotically standard.
\end{thm}

The zero-circles of a near-symplectic form are not all the same. They come in two ``types'' depending on the behavior of $\omega$ near them (see the author's paper \cite{Gerig:taming}), called \textit{untwisted} and \textit{twisted}. By work of Luttinger (details by Perutz and Taubes \cite{Perutz:near,Taubes:Luttinger}), $\omega$ can be modified so that $\omega^{-1}(0)$ has any positive number of components, and that a twisted zero-circle can be traded for two untwisted zero-circles. But as noted by Gompf (see Perutz \cite{Perutz:near}*{Theorem 1.8}) in the case of closed 4-manifolds, the parity of the number of untwisted zero-circles is a priori determined by the cohomology of the closed 4-manifold. It follows that an asymptotically standard near-symplectic form on $X^\ast$ must have an even number of untwisted zero-circles: The reason is that we can glue this 2-form into any near-symplectic closed 4-manifold (by the Darboux theorem) to build another near-symplectic form, so it must preserve the parity of the number of untwisted zero-circles of the original near-symplectic form.

\begin{question}
\label{question}
If $X^\ast$ is symplectic then an asymptotically standard near-symplectic form $\omega$ can be modified to remove its zero-set, but can this be done locally? When $\omega^{-1}(0)$ is a single twisted circle, is there a unique pseudoholomorphic disk bounding it? When $\omega^{-1}(0)$ is two circles (both of the same type), is there a unique pseudoholomorphic cylinder bounding it?
\end{question}

\begin{example}
\label{ex:birth}
The standard symplectic form on $\CC^2$ with complex coordinates $(z,w)$ is $\omega_{\op{std}}=\frac{i}{2}(dz\wedge d\bar z+dw\wedge d\bar w)$. A version of Luttinger's birth model \cite{LuttingerSimpson} is a family of near-symplectic forms on $\CC^2$
$$\omega_\e=\frac{i}{2}\left[(-\e+|z|^2-|w|^2)(dz\, d\bar z+dw\, d\, w)+(Rw-\bar z\bar w)dz\, dw-(R\bar w-zw)d\bar z\, d\bar w\right]$$
with parameter $\e\in\RR$ and fixed $R\gg1$. For $0<\e\ll1$ there are two zero-circles, one of which dies as $\e$ tends to 0, namely $Z_\e=\left\{(z,0)\in\CC^2\;|\;|z|^2=\e\right\}$. For fixed such $\e$ we can modify a constant multiple of $\omega_\e$ on the complement of the radius $\sqrt{2\e}$ ball about the origin (which contains $Z_\e$), in such a way that the other zero-circle is destroyed and that it agrees with $\omega_{\op{std}}$ on the complement of the radius $32\sqrt{2\e}$ ball (see \cite{Taubes:Luttinger}*{\S2} for details). The result is an asymptotically standard near-symplectic form with a single twisted zero-circle, and $\lbrace(z,0)\in\CC^2\;|\;|z|^2\le\e\rbrace$ is an $i$-holomorphic disk bounding $Z_\e$.
\end{example}

With respect to Question~\ref{question}, Taubes suggested (on general near-symplectic 4-manifolds) that the existence of certain pseudoholomorphic cylinders between zero-circles may be used to cancel the zero-circles, analogous to the ``Morse cancellation lemma'' for certain gradient flowlines between critical points of a Morse function \cite{Taubes:geom, Taubes:geomICM, Taubes:broken}. If successful this would reduce Conjecture~\ref{conj} to the Schoenflies conjecture. But the methodology taken in this paper is the opposite, we suggest that the existence of pseudoholomorphic curves may prevent the cancellation of zero-circles. That is, we would like to build invariants of $X$ by counting pseudoholomorphic curves in a completion of $X^\ast-\omega^{-1}(0)$ for any asymptotically standard near-symplectic form $\omega$. The significance is that they would give obstructions to removing the zero-circles of $\omega$, and may detect counterexamples to Conjecture~\ref{conj}.\footnote{In this direction, if we are given an explicit handlebody decomposition of $X$ involving a single 4-handle and no 3-handles, then \cite{Scott:thesis} builds explicit near-symplectic 2-forms on $X^\ast$ based on 2-handle data. One stumbling block here is the Andrews--Curtis conjecture in group theory \cite{AndrewsCurtis}, which if true would imply that any homotopy sphere given as a handlebody without 3-handles is standard.}

We build such an invariant by mimicking the construction of the Gromov invariants of closed near-symplectic 4-manifolds (by the author \cite{Gerig:taming, Gerig:Gromov}), in turn using embedded contact homology and its known isomorphism with monopole Floer homology. Unfortunately, it is not sensitive enough. We clarify and summarize this as the following main result.

\bigskip\noindent
\textbf{Main Result 1}. \textit{Fix an asymptotically standard near-symplectic form $\omega$ on $X^\ast$ which has no twisted zero-circles. For a suitable neighborhood $\nN$ of $\omega^{-1}(0)$ such that $X^\ast-\nN$ is a symplectic manifold with contact boundary (as described in Lemma~\ref{lem:nbhd}), there exists a well-defined element $Gr_{X,\omega}$ in the embedded contact homology $ECH_\ast(\partial\nN)$ as described in Theorem~\ref{thm:class}, obtained by a suitable count of punctured pseudoholomorphic curves in a completion of $X^\ast-\nN$ which are asymptotic to specific Reeb orbits in $\partial\nN$. There is also a relative Seiberg--Witten invariant associated with $X^\ast$ as described in Theorem~\ref{thm:SW}, and its value does not depend on $X^\ast$. Then $Gr_{X,\omega}$ is identified with this relative Seiberg--Witten invariant (Theorem~\ref{thm:Gr}), so it is not able to detect potentially exotic 4-spheres.}

\begin{remark}
\label{rmk:orient}
We work over $\ZZ/2$ when relating $Gr_{X,\omega}$ to Seiberg--Witten theory, but we expect the same argument to apply over $\ZZ$ once we figure out how the maps in Proposition~\ref{prop:multivalued} intertwine the ``coherent orientations'' of the moduli spaces (see \cite{Gerig:Gromov} for further clarification on this matter).
\end{remark}

Here is an outlook on this null-result. The near-symplectic ECH-type invariant $Gr_{X,\omega}$ is seen to be inherently related to Seiberg--Witten theory, and subsequently not helpful to distinguish 4-spheres. Now there is also the broader machinery of SFT (symplectic field theory) to consider, which subsumes ECH in some sense. So we may try to define an analogous near-symplectic SFT-type invariant -- but this also turns out not to be helpful (see Section~\ref{Symplectic field theory}). These facts may suggest that the existence of pseudoholomorphic curves won't obstruct the removal of zero-circles and will instead be useful in the removal of them; otherwise we have to consider more intricate moduli spaces of pseudoholomorphic curves.

The proof that $Gr_{X,\omega}$ is independent of the smooth structure on $X$ may be rephrased in terms of monopole Floer cobordism maps. Given either a homotopy 4-sphere or any closed smooth 4-manifold, by removing two disjoint small 4-balls from it we obtain a cobordism from $S^3$ to $S^3$. There are then three induced cobordism maps on monopole Floer homology corresponding to the three flavors $\Hto,\Hfrom,\overline{HM}$. When $b^2_+>0$ these maps are all trivial (see Kronheimer \& Mrowka \cite{KM:book}*{\S27.3}) and for a homotopy 4-sphere these maps are all the identity, so they are not viable options to distinguish smooth structures.\footnote{When $b^2_+>0$ there is also a ``mixed'' cobordism map $\overrightarrow{HM}_\ast:\Hfrom_\ast(S^3)\to\Hto_\ast(S^3)$ which may be used to recover the Seiberg--Witten invariants of the 4-manifold, but it is not defined for homotopy 4-spheres.} But we go further and compute $Gr_{X,\omega}$ (or really the corresponding relative Seiberg--Witten invariant) to get the following existence result.

\bigskip\noindent
\textbf{Main Result 2}. \textit{For any homotopy 4-sphere $X$ and given any asymptotically standard near-symplectic form $\omega$ on $X^\ast$ which has no twisted zero-circles, there always exists a pseudoholomorphic curve in the complement of $\omega^{-1}(0)$ in $X^\ast$ which homologically bounds $\omega^{-1}(0)$ in the sense that it has intersection number 1 with every linking 2-sphere of $\omega^{-1}(0)$, and whose components are either embedded or a disjoint cover of an embedded plane.}

\bigskip
This 2nd main result is a corollary to the 1st main result, using Theorem~\ref{thm:class} to explicitly spell out the asymptotics of the relevant curves and noting that $Gr_{\RR^4,\omega_{\op{std}}}\equiv1\mod2$. This leaves the question of how to make use of such a curve.

\subsection*{Acknowledgements}
The author thanks Robert Gompf, Michael Hutchings, Peter Kronheimer, Tomasz Mrowka, Clifford Taubes, and Boyu Zhang. As a historical remark, Hutchings discussed related ideas with Yakov Eliashberg during his postdoctoral position at Stanford in 2002, which are subsequently related to the ideas of Taubes. This material is based upon work supported by the National Science Foundation under Award \#1803136.

\section{Brief review of pseudoholomorphic curve theory}
\label{Brief review of pseudoholomorphic curve theory}

We introduce most of the terminology and notations that appear later on with respect to symplectic geometry. More details are found in \cite{Hutchings:lectures}.

\subsection{Orbits}
\label{Orbits}

Let $(Y,\lambda)$ be a closed contact 3-manifold, oriented by $\lambda\wedge d\lambda>0$, and let $\xi=\ker\lambda$ be its contact structure. With respect to the \textit{Reeb vector field} $R$ determined by $d\lambda(R,\cdot)=0$ and $\lambda(R)=1$, a Reeb orbit is a map $\gamma:\RR/T\ZZ\to Y$ for some $T>0$ with $\gamma'(t)=R(\gamma(t))$, modulo reparametrization. A given Reeb orbit is \textit{nondegenerate} if the linearization of the Reeb flow around it does not have 1 as an eigenvalue, in which case the eigenvalues are either on the unit circle (such $\gamma$ are \textit{elliptic}) or on the real axis (such $\gamma$ are \textit{hyperbolic}). Assume from now on that $\lambda$ is \textit{nondegenerate}, i.e. all Reeb orbits are nondegenerate. 

An \textit{orbit set} is a finite set of pairs $\Theta=\lbrace(\Theta_i,m_i)\rbrace$ where the $\Theta_i$ are distinct embedded Reeb orbits and the $m_i$ are positive integers (which may be empty). An orbit set is $\textit{admissible}$ if $m_i=1$ whenever $\Theta_i$ is hyperbolic. Its homology class is defined by
$$[\Theta]:=\sum_im_i[\Theta_i]\in H_1(Y;\ZZ)$$
For a given $\Gamma\in H_1(Y;\ZZ)$, the \textit{ECH chain complex} $ECC_\ast(Y,\lambda,J,\Gamma)$ is freely generated over $\ZZ/2$ by admissible orbit sets representing $\Gamma$. The differential $\partial_\text{ECH}$ will be defined momentarily.

\subsection{Curves}
\label{Curves}

Given two contact manifolds $(Y_\pm,\lambda_\pm)$, possibly disconnected or empty, a \textit{strong symplectic cobordism} from $(Y_+,\lambda_+)$ to $(Y_-,\lambda_-)$ is a compact symplectic manifold $(X,\omega)$ with oriented boundary
$$\partial X=Y_+\sqcup -Y_-$$
such that $\omega|_{Y_\pm}=d\lambda_\pm$.
We can always find neighborhoods $N_\pm$ of $Y_\pm$ in $X$ diffeomorphic to $(-\e,0]\times Y_+$ and $[0,\e)\times Y_-$, such that $\omega|_{N_\pm}=d(e^{\pm s}\lambda_\pm)$ where $s$ denotes the coordinate on $(-\e,0]$. We then glue symplectization ends to $X$ to obtain the \textit{completion}
$$\overline X:=\big((-\infty,0]\times Y_-\big)\cup_{Y_-}X\cup_{Y_+}\big([0,\infty)\times Y_+\big)$$
of $X$, a noncompact symplectic 4-manifold whose symplectic form is also denoted by $\omega$. We will also use the notation $\overline X$ to denote the symplectization $\RR\times Y$ of $(Y,\lambda)$, with $\omega=d(e^s\lambda)$.

An almost complex structure $J$ on a symplectization $(\RR\times Y,d(e^s\lambda))$ is \textit{symplectization-admissible} if it is $\RR$-invariant; $J(\partial_s)=R$; and $J(\xi)\subseteq\xi$ such that $d\lambda(v,Jv)\ge 0$ for $v\in\xi$. An almost complex structure $J$ on the completion $\overline X$ is \textit{cobordism-admissible} if it is $\omega$-compatible on $X$ and agrees with symplectization-admissible almost complex structures on the ends $[0,\infty)\times Y_+$ and $(-\infty,0]\times Y_-$.

\bigskip
Given a cobordism-admissible $J$ on $\overline X$ and orbit sets $\Theta^+=\lbrace(\Theta^+_i,m^+_i)\rbrace$ in $Y_+$ and $\Theta^-=\lbrace(\Theta^-_j,m^-_j)\rbrace$ in $Y_-$, a \textit{$J$-holomorphic curve $\cC$ in $\overline X$ from $\Theta^+$ to $\Theta^-$} is defined as follows. It is a $J$-holomorphic map $\cC\to \overline X$ whose domain is a possibly disconnected punctured compact Riemann surface, defined up to composition with biholomorphisms of the domain, with positive ends of $\cC$ asymptotic to covers of $\Theta^+_i$ with total multiplicity $m^+_i$, and with negative ends of $\cC$ asymptotic to covers of $\Theta^-_j$ with total multiplicity $m^-_j$ (see \cite{Hutchings:lectures}*{\S3.1}). The moduli space of such curves up to equivalence is denoted by $\mM(\Theta^+,\Theta^-)$, where two such curves are considered equivalent if they represent the same current in $\overline X$, and in the case of a symplectization $\overline X=\RR\times Y$ the equivalence includes translation of the $\RR$-coordinate. An element $\cC\in\mM(\Theta^+,\Theta^-)$ can thus be viewed as a finite set of pairs $\lbrace(C_k,d_k)\rbrace$ or formal sum $\sum d_kC_k$, where the $C_k$ are distinct irreducible somewhere-injective\footnote{It's well known that a connected $J$-holomorphic curve is either somewhere-injective or multiply covered.} $J$-holomorphic curves and the $d_k$ are positive integers. 

Let $H_2(X,\Theta^+,\Theta^-)$ be the set of relative 2-chains $\Sigma$ in $X$ such that
$$\partial \Sigma=\sum_im^+_i\Theta^+_i-\sum_jm^-_j\Theta^-_j$$
modulo boundaries of 3-chains. It is an affine space over $H_2(X;\ZZ)$, and every curve $\cC$ defines a relative class $[\cC]\in H_2(X,\Theta^+,\Theta^-)$.

\subsection{Homology}
\label{Homology}

The \textit{ECH index} $I(\cC)$ of a current $\cC\in\mM(\Theta^+,\Theta^-)$ is an integer depending only on its relative class in $H_2(X,\Theta^+,\Theta^-)$, and is the local expected dimension of this moduli space of $J$-holomorphic currents (see \cite{Hutchings:lectures}*{\S3}). Denote by $\mM_I(\Theta^+,\Theta^-)$ the subset of elements in $\mM(\Theta^+,\Theta^-)$ that have ECH index $I$. 

Given admissible orbit sets $\Theta^\pm$ of $(Y,\lambda)$, the coefficient $\langle\partial_\text{ECH}\Theta^+,\Theta^-\rangle\in\ZZ/2$ is the count (modulo 2) of elements in $\mM_1(\Theta^+,\Theta^-)$ on the symplectization $\overline X=\RR\times Y$. If $J$ is generic then $\partial_\text{ECH}$ is well-defined and $\partial^2_\text{ECH}=0$. The resulting homology is independent of the choice of $J$, depends only on $\xi$ and $\Gamma$, and is denoted by $ECH_\ast(Y,\xi,\Gamma)$. The grading is by homotopy classes of oriented 2-plane fields on $Y$, and there is a transitive $\ZZ$-action on this grading set (see \cite{Hutchings:revisited}*{\S 3}).

\subsection{L-flat approximations}
\label{L-flat approximations}

The symplectic action of an orbit set $\Theta=\lbrace(\Theta_i,m_i)\rbrace$ is defined by
$$\aA(\Theta):=\sum_im_i\int_{\Theta_i}\lambda$$
The symplectic action induces a filtration on the ECH chain complex. For a positive real number $L$, the $L$-\textit{filtered ECH} is the homology of the subcomplex $ECC_\ast^L(Y,\lambda,J,\Gamma)$ spanned by admissible orbit sets of action less than $L$. The ordinary ECH is recovered by taking the direct limit over $L$, via maps induced by inclusions of the filtered chain complexes.

For a fixed $L>0$ it is convenient (and possible) to modify $\lambda$ and $J$ on small tubular neighborhoods of all Reeb orbits of action less than $L$, in order to relate $J$-holomorphic curves to Seiberg--Witten theory most easily. The desired modifications of $(\lambda,J)$ are called \textit{L-flat approximations}, and were introduced by Taubes in \cite{Taubes:ECH=SWF1}*{Appendix}. They induce isomorphisms on the $L$-filtered ECH chain complex, and the key fact here is that $L$-flat orbit sets are in bijection with Seiberg--Witten solutions of ``energy'' less than $2\pi L$ (see Section~\ref{Taubes' isomorphisms}).

\section{Gromov invariant for homotopy 4-spheres}
\label{Gromov invariant for homotopy 4-spheres}

Fix an asymptotically standard near-symplectic form $\omega$ on $X^\ast$ having $N\ge0$ untwisted zero-circles ($N$ is even) and no twisted zero-circles, which exists by Theorem~\ref{LuttingerTaubes}.

Let $X^o$ denote the compact submanifold with $S^3$ boundary such that $\omega$ is standard on $X^\ast-X^o$ (view $X^o$ as the closure of the manifold obtained from $X$ by removing a standard 4-ball) and such that $\omega|_{\partial X^o}=\omega_{\op{std}}|_{S^3}=d\lambda_{\op{std}}$, where $\lambda_{\op{std}}$ denotes a generic $C^\infty$-small perturbation of the tautological 1-form $\frac12(x_1dy_1-y_1dx_1+x_2dx_2-y_2 dx_2)$ on $S^3$ (so $\lambda_{\op{std}}$ is a nondegenerate contact form).

Let $\nN$ denote the union of arbitrarily small tubular neighborhoods of all components of $\omega^{-1}(0)\subset X^\ast$, so it is diffeomorphic to the disjoint union of $N$ copies of $S^1\times B^3$. The orientation of each zero-circle is pinned down using $\nabla\omega$ and the orientation of $X^\ast$ according to the orientation conventions specified in \cite{Gerig:taming}*{\S1.2, \S3.1}, and in particular gives the identification of $H_1(S^1\times S^2)$ with $\ZZ$ on each component of $\partial\nN$ such that $[S^1\times\lbrace pt\rbrace]\mapsto-1$. With that said, consider the relative homology class
$$A_\1\in H_2(X^o-\nN,\partial\nN;\ZZ)\cong H_1(\partial\nN;\ZZ)$$
uniquely specified by
$$\partial A_\1=\1:=(1,\ldots,1)\in\ZZ^N\cong H_1(\partial\nN;\ZZ)$$
under the long exact sequence for the pair $(X^o-\nN,\partial\nN)$.\footnote{The map $H_2(\partial\nN;\ZZ)\to H_2(X^o-\nN;\ZZ)$ in the long exact sequence is an isomorphism, as seen using the Mayer-Vietoris sequence for the union $(X^o-\nN)\cup\nN=X^o$.} We have the following useful choice of $\nN$ granted by \cite{Gerig:taming}*{Lemma 1.6}.

\begin{lemma}
\label{lem:nbhd}
The tubular neighborhood $\nN$ may be chosen in such a way that $(X^o-\nN,\omega)$ is a strong symplectic cobordism from $(S^3,\lambda_{\op{std}})$ to $N$ copies of $(S^1\times S^2,\lambda_{\op{ns}})$. Here, $\lambda_{\op{ns}}$ is an overtwisted contact form whose orbits of symplectic action less than $\rho(A_\1)$ are all $\rho(A_\1)$-flat and are either hyperbolic or $\rho(A_\1)$-positive elliptic.
\end{lemma}

We have yet to define the quantity $\rho(A_\1)\in\RR$ and the adjective ``$\rho(A_\1)$-positive'' (but the notion of ``flatness'' was clarified in Section~\ref{L-flat approximations}). In order to obtain well-defined counts of pseudoholomorphic curves in $X^o-\nN$ which represent the relative class $A_\1$, we need to ensure a uniform bound on their energy as well as a bound on the symplectic action of their orbit sets, and we need to guarantee transversality of the relevant moduli spaces of curves (specifically, to rule out negative ECH index curves). The quantity $\rho(A_\1)$ provides the bounds, and the adjective ``$\rho(A_\1)$-positive'' ultimately ensures transversality -- we will not define this adjective here (but see \cite{Gerig:taming}*{\S3.2}).

\begin{defn}
$$\rho(A_\1):=\int_\Sigma\omega-\int_{\partial\Sigma\cap S^3}\lambda_{\op{std}}+\sum_{k=1}^N\int_{\partial\Sigma\cap(S^1\times S^2)_k}\lambda_{\op{ns}}$$
where $u:\Sigma\to X^o-\nN$ is any smooth map representing $A_\1\in H_2(X^o-\nN,\partial\nN;\ZZ)$, such that $\Sigma$ is a compact oriented smooth surface with boundary satisfying $u(\partial \Sigma)\subset \partial(X^o-\nN)$.
\end{defn}

Note that $\rho(A_\1)$ need not be $0$ even though $\omega$ is an exact 2-form on $X^o-\nN$, because a primitive 1-form $\nu$ (such that $\omega=d\nu$) need not agree with our contact forms on any copy of $S^1\times S^2$. We can only arrange that $\nu|_{S^3}=\lambda_{\op{std}}$ since $H^1(S^3;\RR)=0$.  In the literature, the cobordism $(X^o-\nN,\omega)$ is called \textit{weakly exact}.

Fix a cobordism-admissible almost complex structure $J$ on the completion $(\overline{X^o-\nN},\omega)$, as specified in Section~\ref{Curves}. We now present counts of $J$-holomorphic curves which assemble into a well-defined element of tensor products of copies of $ECH_\ast(S^1\times S^2,\xi_{\op{ns}},1)$. We recall from Section~\ref{Homology} that the set $\mM_0(\varnothing,X^o-\nN,\Theta)$ consists of $J$-holomorphic currents which have negative ends asymptotic to $\Theta$ (and no positive ends). Let $\mM_0(\varnothing,\Theta;A_\1)$ denote the subset of elements in $\mM_0(\varnothing,X^o-\nN,\Theta)$ which represent the class $A_\1$. Define the chain
\begin{equation}
\label{eqn:chain}
\sum_\Theta\#\mM_0(\varnothing,\Theta_\1;A)\cdot\Theta\in\bigotimes_{k=1}^NECC_\ast(S^1\times S^2,\lambda_{\op{ns}},1)
\end{equation}
where $\Theta$ indexes over the admissible orbit sets (representing $\1$) -- the implicit fact that each moduli space $\mM_0(\varnothing,\Theta_\1;A)$ is a finite set of points is subsumed in the following theorem.

\begin{thm}
\label{thm:class}
For generic $J$, the chain~\eqref{eqn:chain} induces a well-defined element
$$Gr_{X,\omega}\in\bigotimes_{k=1}^NECH_\ast(S^1\times S^2,\xi_{\op{ns}},1)$$
concentrated in a single absolute grading.
\end{thm}

\begin{proof}
The fact that the chain is a cycle follows \cite{Gerig:taming} verbatim, using the conditions granted by Lemma~\ref{lem:nbhd}. The reason we may copy the arguments in \cite{Gerig:taming} is that there are no positive ends of the relevant pseudoholomorphic curves, so we may treat $(X^o-\nN,\omega)$ as if it were a symplectic cobordism with only negative boundary components. Note that we do not take a \textit{weighted} count of elements of each $\mM_0(\varnothing,\Theta_\1;A)$ when defining $Gr_{X,\omega}$, contrasted with that in \cite{Gerig:taming}, because there are no non-constant \textit{closed} pseudoholomorphic curves in $X^o-\nN$.
\end{proof}

\begin{remark}
As explained in \cite{Gerig:taming}, $ECH_\ast(S^1\times S^2,\xi_{\op{ns}},\Gamma)=0$ for $\Gamma\ne1$. This is why we only consider the relative class $A_\1$ satisfying $\partial A_\1|_{S^1\times S^2}=1$.
\end{remark}

\begin{remark}
We can define $Gr_{X,\omega}$ over $\ZZ$ by introducing orientations, but see Remark~\ref{rmk:orient}.
\end{remark}

We would hope that $Gr_{X,\omega}$ only depends on $X$. In particular, part of this project would involve showing that $Gr_{X,\omega}$ does not depend on the choice of near-symplectic form, by analyzing the moduli spaces of pseudoholomorphic curves as $\omega$ deforms and $X^\ast-\omega^{-1}(0)$ changes topological type. In this regard, if $\omega$ is symplectic (i.e. $N=N^\sigma=0$) then $Gr_{X,\omega}$ lives in $ECH_0(\varnothing,0,0)\cong\ZZ/2$ generated by the empty orbit set, or alternatively we can remove a Darboux ball $D$ from $X^\ast$ to view $Gr_{X,\omega}$ in $ECH_0(S^3,\xi_{\op{std}},0)\cong\ZZ/2$ generated by the empty orbit set. Then $Gr_{X,\omega}=1$ because only the empty curve exists, representing $0\in H_2(X^\ast-D,S^3;\ZZ)\cong H_2(X^\ast;\ZZ)\cong 0$.

On hindsight it turns out that $Gr_{X,\omega}$ does not depend on $X$ (let alone $\omega$ and $J$), because ECH is related to Seiberg--Witten theory. The rest of this paper will clarify what this means.

\section{Brief review of gauge theory}
\label{Brief review of gauge theory}

We introduce most of the terminology and notations that appear later on with respect to Seiberg--Witten theory. More details are found in \cite{KM:book, HutchingsTaubes:Arnold2}.

\subsection{Contact 3-manifolds}
\label{Contact 3-manifolds}

Let $(Y,\lambda)$ be a closed oriented connected contact 3-manifold, and choose an almost complex structure $J$ on $\xi$ that induces a symplectization-admissible almost complex structure on $\RR\times Y$. There is a compatible metric $g$ on $Y$ such that $|\lambda|=1$ and $\ast\lambda=\frac12d\lambda$, with $g(v,w)=\frac12d\lambda(v,Jw)$ for $v,w\in\xi$.

View a spin-c structure $\fs\in\Spinc(Y)$ on $Y$ as an isomorphism class of a pair $(\SS,\cl)$ consisting of a rank 2 Hermitian vector bundle $\SS\to Y$ (the \textit{spinor bundle}) and Clifford multiplication $\cl:TY\to\End(\SS)$. The contact structure $\xi$ (and more generally, any oriented 2-plane field on $Y$) picks out a canonical spin-c structure $\fs_\xi=(\SS_\xi,\cl)$ with $\SS_\xi=\underline{\CC}\oplus\xi$, where $\underline{\CC}\to Y$ denotes the trivial line bundle, and $\cl$ is defined as follows. Given an oriented orthonormal frame $\lbrace e_1,e_2,e_3\rbrace$ for $T_yY$ such that $\lbrace e_2,e_3\rbrace$ is an oriented orthonormal frame for $\xi_y$, then in terms of the basis $(1,e_2)$ for $\SS_\xi$,
$$\cl(e_1)=\bigl( \begin{smallmatrix}
i&0\\ 0&-i
\end{smallmatrix} \bigr),\indent \cl(e_2)=\bigl( \begin{smallmatrix}
0&-1\\ 1&0
\end{smallmatrix} \bigr),\indent \cl(e_3)=\bigl( \begin{smallmatrix}
0&i\\ i&0
\end{smallmatrix} \bigr)$$
There is then a canonical isomorphism
$$H^2(Y;\ZZ)\to\Spinc(Y)$$
where the 0 class corresponds to $\fs_\xi$. Specifically, there is a canonical decomposition $\SS=E\oplus \xi E$ into $\pm i$ eigenbundles of $\cl(\lambda)$, where $E\to Y$ is the complex line bundle corresponding to a given class in $H^2(Y;\ZZ)$.

A \textit{spin-c connection} is a connection $\bA$ on $\SS$ which is compatible with Clifford multiplication in the sense of \cite{HutchingsTaubes:Arnold2}*{Equation 18}. Such a connection is equivalent to a Hermitian connection (also denoted by $\bA$) on $\det\SS$, and determines a \textit{Dirac operator}
$$D_\bA:\Gamma(\SS)\stackrel{\nabla_\bA}{\longrightarrow}\Gamma(T^\ast Y\otimes\SS)\stackrel{\cl}{\longrightarrow}\Gamma(\SS)$$
With respect to the decomposition $\SS=E\oplus\xi E$, there is a unique connection $A_\xi$ on $\xi$ such that its Dirac operator kills the section $(1,0)\in\Gamma(\SS_\xi)$, and there is a canonical decomposition
$$\bA=A_\xi+2A$$
on $\det\SS=\xi E^2$ with Hermitian connection $A$ on $E$. The gauge group $C^\infty(Y,S^1)$ acts on a given pair $(A,\psi)$ by
$$u\cdot(A,\psi)=(A-u^{-1}du,u\psi)$$
In this paper, a \textit{configuration} $\fc$ refers to a gauge-equivalence class of such a pair.

Fix a suitably generic exact 2-form $\mu\in\Omega^2(Y)$ as described in \cite{HutchingsTaubes:Arnold2}*{\S 2.2}, and a positive real number $r\in\RR$. A configuration $\fc$ solves \textit{Taubes' perturbed Seiberg--Witten equations} when
\begin{equation}
\label{SW3}
D_\bA\psi=0,\indent\indent\ast F_A=r(\tau(\psi)-i\lambda)-\frac12\ast F_{A_\xi}+i\ast\mu
\end{equation}
where $F_{A_\xi}$ is the curvature of $A_\xi$ and $\tau:\SS\to iT^\ast Y$ is the quadratic bundle map which sends $\psi\in\SS$ to the covector $\langle\cl(\cdot)\psi,\psi\rangle$. An appropriate change of variables recovers the usual Seiberg--Witten equations (with perturbations) that appear in \cite{KM:book}. 

\begin{remark}
We have suppressed additional ``abstract tame perturbations'' to these equations required to obtain transversality of the moduli spaces of its solutions (see \cite{KM:book}*{\S 10}), because they do not interfere with the analysis presented in this paper. This is further clarified in \cite{HutchingsTaubes:Arnold2}*{\S 2.1} and \cite{Taubes:ECH=SWF1}*{\S3.h Part 5}, where the same suppression occurs.
\end{remark}

Denote by $\fM(Y,\fs)$ the set of solutions to~\eqref{SW3}, called \textit{(SW) monopoles}. A solution is \textit{reducible} if its $\Gamma(\SS)$-component vanishes, and is otherwise \textit{irreducible}. The monopoles freely generate the monopole Floer chain complex $\Cfrom^\ast(Y,\lambda,\fs,J,r)$. The chain complex differential will not be reviewed here. Denote by $\Cfrom_L^\ast(Y,\lambda,\fs,J,r)$ the submodule generated by irreducible monopoles $\fc$ with energy
$$E(\fc):=i\int_Y\lambda\wedge F_A<2\pi L$$
When $r$ is sufficiently large, $\Cfrom_L^\ast(Y,\lambda,\fs,J,r)$ is a subcomplex of $\Cfrom^\ast(Y,\lambda,\fs,J,r)$ and its homology $\Hfrom_L^\ast(Y,\lambda,\fs,J,r)$ is well-defined and independent of $r$ and $\mu$. Taking the direct limit over $L>0$, we recover the ordinary $\Hfrom^\ast(Y,\fs)$ in \cite{KM:book} which is independent of $\lambda$ and $J$. The grading is also by homotopy classes of oriented 2-plane fields on $Y$, in Section~\ref{Homology} (this convention is opposite to that used in \cite{KM:book}).

\subsection{Symplectic cobordisms}
\label{Symplectic cobordisms}

Let $(X,\omega)$ be a strong symplectic cobordism between (possibly disconnected or empty) closed oriented contact 3-manifolds $(Y_\pm,\lambda_\pm)$. Due to the choice of metric $g_\pm$ on $Y_\pm$ in Section~\ref{Contact 3-manifolds} (and following \cite{HutchingsTaubes:Arnold2}*{\S4.2}), we do not extend $\omega$ over $\overline X$ using $d(e^s\lambda_\pm)$ on the ends $(-\infty,0]\times Y_-$ and $[0,\infty)\times Y_+$. Instead, we extend $\omega$ using $d(e^{2s}\lambda_\pm)$ as follows. Fix a smooth increasing function $\phi_-:(-\infty,\e]\to(-\infty,\e]$ with $\phi_-(s)=2s$ for $s\le\frac{\e}{10}$ and $\phi_-(s)=s$ for $s>\frac{\e}{2}$, and fix a smooth increasing function $\phi_+:[-\e,\infty)\to[-\e,\infty)$ with $\phi_+(s)=2s$ for $s\ge-\frac{\e}{10}$ and $\phi_+(s)=s$ for $s\le-\frac{\e}{2}$, where $\e>0$ is such that $\omega=d(e^s\lambda_\pm)$ on the $\e$-collars of $Y_\pm$. Then the desired extension is
\begin{equation*}
\tilde\omega:=\begin{cases}
    d(e^{\phi_-}\lambda_-) & \text{on}\;\; (-\infty,\e]\times Y_-\\
    \omega & \text{on}\;\; X\setminus\Big(\big([0,\e]\times Y_-\big)\cup\big([-\e,0]\times Y_+\big)\Big)\\
    d(e^{\phi_+}\lambda_+) & \text{on}\;\; [-\e,\infty)\times Y_+
  \end{cases}
\end{equation*}
Now choose a cobordism-admissible almost complex structure $J$ on $(\overline X,\tilde\omega)$. Following \cite{HutchingsTaubes:Arnold2}*{\S4.2}, we equip $\overline X$ with a particular metric $g$ so that it agrees with the product metric with $g_\pm$ on the ends $(-\infty,0]\times Y_-$ and $[0,\infty)\times Y_+$ and so that $\tilde\omega$ is self-dual. Finally, define
$$\widehat\omega:=\sqrt{2}\tilde\omega/|\tilde\omega|_g$$
and note that $J$ is still cobordism-admissible.

The 4-dimensional gauge-theoretic scenario is analogous to the 3-dimensional scenario. View a spin-c structure $\fs$ on $X$ as an isomorphism class of a pair $(\SS,\cl)$ consisting of a Hermitian vector bundle $\SS=\SS_+\oplus\SS_-$, where the spinor bundles $\SS_\pm$ have rank 2, and Clifford multiplication $\cl:TX\to\End(\SS)$ such that $\cl(v)$ exchanges $\SS_+$ and $\SS_-$ for each $v\in TX$. The set $\Spinc(X)$ of spin-c structures is an affine space over $H^2(X;\ZZ)$, and we denote by $c_1(\fs)$ the first Chern class of $\det\SS_+=\det\SS_-$. A spin-c connection on $\SS$ is equivalent to a Hermitian connection $\bA$ on $\det\SS_+$ and defines a Dirac operator $D_\bA:\Gamma(\SS_\pm)\to\Gamma(\SS_\mp)$.

A spin-c structure $\fs$ on $X$ restricts to a spin-c structure $\fs|_{Y_\pm}$ on $Y_\pm$ with spinor bundle $\SS_{Y_\pm}:=\SS_+|_{Y_\pm}$ and Clifford multiplication $\cl_{Y_\pm}(\cdot):=\cl(v)^{-1}\cl(\cdot)$, where $v$ denotes the outward-pointing unit normal vector to $Y_+$ and the inward-pointing unit normal vector to $Y_-$. There is a canonical way to extend $\fs$ over $\overline X$, and the resulting spin-c structure is also denoted by $\fs$. There is a canonical decomposition $\SS_+=E\oplus K^{-1}E$ into $\mp2i$ eigenbundles of $\cl_+(\widehat\omega)$, where $K$ is the canonical bundle of $(\overline X,J)$ and $\cl_+:\bigwedge^2_+T^\ast\overline X\to\End(\SS_+)$ is the projection of Clifford multiplication onto $\End(\SS_+)$. This agrees with the decomposition of $\SS_{Y_\pm}$ on the ends of $\overline X$.

The symplectic form $\omega$ picks out the canonical spin-c structure $\fs_\omega=(\SS_\omega,\cl)$, namely that for which $E$ is trivial, and the $H^2(X;\ZZ)$-action on $\Spinc(X)$ becomes a canonical isomorphism. There is a unique connection $A_{K^{-1}}$ on $K^{-1}$ such that its Dirac operator annihilates the section $(1,0)\in\Gamma((\SS_\omega)_+)$, and we henceforth identify a spin-c connection with a Hermitian connection $A$ on $E$.

In this paper, a \textit{configuration} $\fd$ refers to a gauge-equivalence class of a pair $(\bA,\Psi)$ under the gauge group $C^\infty(X,S^1)$-action. A connection $\bA$ on $\det\SS_+$ is in \textit{temporal gauge} on the ends of $\overline X$ if
$$\nabla_\bA=\frac{\partial}{\partial s}+\nabla_{\bA(s)}$$
on $(-\infty,0]\times Y_-$ and $Y_+\times[0,\infty)$, where $\bA(s)$ is a connection on $\det\SS_{Y_\pm}$ depending on $s$. Connections are placed into temporal gauge by an appropriate gauge transformation.

Fix suitably generic exact 2-forms $\mu_\pm\in\Omega^2(Y_\pm)$ as in Section~\ref{Contact 3-manifolds}, a suitably generic exact 2-form $\mu\in\Omega^2(\overline X)$ that agrees with $\mu_\pm$ on the ends of $\overline X$ (with $\mu_\ast$ denoting its self-dual part), and a positive real number $r\in\RR$. \textit{Taubes' perturbed Seiberg--Witten equations} for a configuration $\fd$ are
\begin{equation}
\label{SW4}
D_\bA\Psi=0,\;\;F^+_A=\frac{r}{2}(\rho(\Psi)-i\widehat\omega)-\frac12F^+_{A_{K^{-1}}}+i\mu_\ast
\end{equation}
where $F_A^+$ is the self-dual part of the curvature of $A$ and $\rho:\SS_+\to\bigwedge^2_+T^\ast X$ is the quadratic bundle map which sends $\Psi\in\SS_+$ to the tensor $-\frac12\big\langle[\cl(\cdot),\cl(\cdot)]\Psi,\Psi\big\rangle$. Similarly to the 3-dimensional equations, there are additional ``abstract tame perturbations'' which have been suppressed (see \cite{KM:book}*{\S24.1}). Given monopoles $\fc_\pm$ on $Y_\pm$, denote by $\fM(\fc_-,X,\fc_+;\fs)$ the set of solutions to~\eqref{SW4} which are asymptotic to $\fc_\pm$ (in temporal gauge on the ends of $\overline X$), called \textit{SW solutions}.

Similarly to ECH, an ``index'' is associated with each SW solution, namely the local expected dimension of the moduli space of SW solutions. Denote by $\fM_k(\fc_-,X,\fc_+;\fs)$ the subset of elements in $\fM(\fc_-,X,\fc_+;\fs)$ that have index $k$.

\subsection{Taubes' isomorphisms}
\label{Taubes' isomorphisms}

With $\ZZ/2$ coefficients, there is a canonical isomorphism of relatively graded modules
$$ECH_\ast(Y,\xi,\Gamma)\cong\Hfrom^{-\ast}(Y,\fs_\xi+\PD(\Gamma))$$
which also preserves the absolute gradings by homotopy classes of oriented 2-plane fields. This isomorphism is constructed on the $L$-filtered chain level. 

\begin{thm}[\cite{Taubes:ECH=SWF1}*{Theorem 4.2}]
\label{thm:generators}
Fix $L>0$ and a generic $L$-flat pair $(\lambda,J)$ on $(Y,\xi)$. Then for $r$ sufficiently large and $\Gamma\in H_1(Y;\ZZ)$, there is a canonical bijection from the set of generators of $ECC^L_\ast(Y,\lambda,\Gamma,J)$ to the set of generators of $\Cfrom_L^\ast(Y,\lambda,\fs_\xi+\PD(\Gamma),J,r)$.
\end{thm}

The image of an admissible orbit set $\Theta$ under this bijection will be denoted by $\fc_\Theta$, and is an irreducible SW monopole that solves Taubes' perturbed Seiberg--Witten equations~\eqref{SW3}.

\subsection{Closed 4-manifolds}
\label{Closed 4-manifolds}

Consider the general scenario of a closed connected oriented Riemannian 4-manifold $(M,g)$. We can recover Seiberg--Witten theory from Section~\ref{Symplectic cobordisms} by taking $(Y_\pm,\lambda_\pm)=(\varnothing,0)$, ignoring the appearance of $\omega$, and thus ignoring the canonical decomposition of $\SS_+$. Then the set of spin-c structures is only an $H^2(M;\ZZ)$-torsor. A configuration $[\bA,\Psi]$ solves \textit{the (perturbed) Seiberg--Witten equations} when
$$D_\bA\Psi=0,\;\;F^+_\bA=\frac12\rho(\Psi)+i\mu$$
where $\mu\in\Omega^2_+(M)$ is a self-dual 2-form. The moduli space of such solutions is denoted by $\fM_\mu(M,\fs)$, and its (virtual) dimension is given by
$$d(\fs):=\frac14\left(c_1(\fs)^2-2\chi(M)-3\sigma(M)\right)$$
where $c_1(\fs)$ denotes the first Chern class of the spin-c structure's positive spinor bundle, $\chi(M)$ denotes the Euler characteristic of $M$, and $\sigma(M)$ denotes the signature of $M$.

\section{Relative Seiberg--Witten}
\label{Relative Seiberg--Witten}

We now continue along with the thread of Section~\ref{Gromov invariant for homotopy 4-spheres} by studying Seiberg--Witten theory on $X^\ast$. Here, we note that the relative class $A_\1$ corresponds to a unique spin-c structure $\fs_0$ on $X^\ast-\nN$ which agrees with the unique spin-c structure on the positive end $[0,\infty)\times S^3$, and it extends over $\nN$ as the unique spin-c structure $\fs_0$ on $X^\ast$.

In the following two subsections~\ref{Cylindrical} and~\ref{AFAK} we choose different metrics and perturbations on the positive end of $X^\ast$ and various extensions of them to the rest of $X^\ast$. The reason we do so is two-fold: Using the setup in Section~\ref{AFAK} we can most easily calculate the dimensions of the relevant Seiberg--Witten moduli spaces, and then we can also calculate the relevant Seiberg--Witten invariants. In the final subsection~\ref{Mixed ends} we relate the two setups (in~\ref{Cylindrical} and~\ref{AFAK}), which will ultimately be used in Section~\ref{Final result and outlook} to fully understand our Gromov invariant from Theorem~\ref{thm:class}.

We always fix the positive scalar curvature metric $g_o$ on $S^3$ which, as in Section~\ref{Contact 3-manifolds}, satisfies $|\lambda_{\op{std}}|=1$ and $\ast\lambda_{\op{std}}=\frac12d\lambda_{\op{std}}$ ($g_o$ is a small perturbation of the round metric).

\subsection{Cylindrical}
\label{Cylindrical}

We may formally apply Section~\ref{Symplectic cobordisms} to the near-symplectic cobordism $(X^o,\omega):(S^3,\lambda_{\op{std}})\to(\varnothing,0)$ by ignoring the canonical decomposition of $\SS_+$ on $\nN$ and extending the 2-form $\widehat\omega|_{X^0-\nN}$ over $\nN$ using a cutoff function that vanishes near $\omega^{-1}(0)$; we abuse notation and denote the resulting 2-form by $\widehat\omega$ still. Then a configuration $[\bA,\Psi]$ on $\overline{X^o}=X^\ast$ with respect to $\fs_0$, which is asymptotic to $\fc\in\fM(S^3,\fs_0)$, solves \textit{the $\widehat\omega$-perturbed Seiberg--Witten equations} when
\begin{equation}
\label{SWcyl}
D_\bA\Psi=0,\;\;F^+_\bA=r(\rho(\Psi)-i\widehat\omega)+2i\mu_\ast
\end{equation}
for a fixed positive real number $r\in\RR$. The moduli space of such solutions is denoted by $\fM^{\op{cyl}}_{\hat\omega}(X^\ast,\fc;\fs_0)$. As a reminder, $X^\ast$ is equipped here with a metric $g$ which is the cylindrical metric $ds^2+g_o$ on $[0,\infty)\times S^3$, and we fixed a suitably generic exact 2-form $\mu\in\Omega^2(X^\ast)$ that agrees on its end with an exact 2-form $\mu\in\Omega^2(S^3)$ from Section~\ref{Contact 3-manifolds}.

\subsection{AFAK}
\label{AFAK}

We now alter the setup from Section~\ref{Cylindrical}. On the positive end $[1,\infty)\times S^3$ of $X^\ast$ we instead equip the conical metric $e^{2s}(ds^2+g_o)$ and symplectic form $\omega_0=\frac12d(e^{2s}\lambda_{\op{std}})$. In the language of \cite{KM:contact}, this end is ``asymptotically flat almost K\"ahler'' (AFAK). We then extend $\omega_0$ to the 2-form $\widehat\omega$ from Section~\ref{Cylindrical} on $X^o$ and denote this extended 2-form on $X^\ast$ by $\overline\omega$. Similarly, we extend the conical metric so that it is cylindrical on $[0,\e]\times S^3$ and agrees with the metric $g$ from Section~\ref{Cylindrical} on $X^o$. Concretely, on the end $[0,\infty)\times S^3$ we have $g=e^{2T}(ds^2+g_o)$ and $\overline\omega=e^{2T}(ds\wedge\lambda_{\op{std}}+\frac12d\lambda_{\op{std}})$, where $T:[0,\infty)\to[0,\infty)$ is a smooth non-decreasing function that has value 0 on $[0,\e]$ and is the identity map on $[2\e,\infty)$. Finally, we fix a suitably generic exact 2-form $\mu\in\Omega^2(X^\ast)$ that has exponential decay on its end and agrees on $[0,\e]\times S^3$ with an exact 2-form $\mu\in\Omega^2(S^3)$ from Section~\ref{Contact 3-manifolds}.

Then a configuration $[\bA,\Psi]$ on $X^\ast$ with respect to $\fs_0$, which is asymptotic to the canonical configuration $[A_\xi,(1,0)]$ on $S^3$ (specified in Section~\ref{Contact 3-manifolds}), solves \textit{the $\overline\omega$-perturbed Seiberg--Witten equations} when
\begin{equation}
\label{SWafak}
D_\bA\Psi=0,\;\;F^+_\bA=r(\rho(\Psi)-i\overline\omega)+2i\mu_\ast
\end{equation}
for a fixed positive real number $r\in\RR$. The moduli space of such solutions is denoted by $\fM^{\op{afak}}_{\overline\omega}(X^\ast,\fs_0)$. This moduli space is of the sort which is used in \cite{KM:contact} to define a ``relative'' Seiberg--Witten invariant of $(X^o,\xi_{\op{std}})$. We take a brief moment to elaborate.

Consider more generally a connected oriented 4-manifold $M$ with nonempty positive boundary, equipped with an oriented contact structure $\xi\to \partial M$ that is compatible with the boundary orientation of $M$. Then the \textit{relative Seiberg--Witten invariant}
$$SW_{M,\xi}(\fs)\in\ZZ$$
is defined for each spin-c structure $\fs\in\Spinc(M)$ that extends the canonical spin-c structure $\fs_\xi$ on $\partial M$, given a choice of ``homology orientation'' of $(M,\xi)$. This was spelled out in \cite{KM:contact} (see also \cite{KMOS:monopoles}*{\S6}), and a discussion about the relation between \cite{KM:contact} and \cite{KM:book} was given in \cite{Taubes:Weinstein2}*{\S4}. Suffice to say that
\begin{equation}
\label{eqn:SWcount}
SW_{X^o,\xi_{\op{std}}}(\fs_0):=\#\fM^{\op{afak}}_{\overline\omega}(X^\ast,\fs_0)
\end{equation}
using the fact that

\begin{lemma}
$\dim \fM^{\op{afak}}_{\overline\omega}(X^\ast,\fs_0)=0$.
\end{lemma}

\begin{proof}
This is the result of \cite{KM:contact}*{Corollary 3.12, Theorem 3.3, Theorem 2.4} because $X^o$ admits a global $\xi_{\op{std}}$-compatible almost complex structure.
\end{proof}

Strictly speaking, since $X^o$ admits a global $\xi_{\op{std}}$-compatible almost complex structure it has a canonical homology orientation by \cite{KM:contact}*{Appendix}, and so the count~\eqref{eqn:SWcount} is taken with signs. We can now state the main theorem of this section.

\begin{thm}
\label{thm:SW}
Given $(X,\omega,\fs_0)$ as above with the canonical homology orientation of $(X^o,\xi_{\op{std}})$, the signed count of points in $\fM^{\op{afak}}_{\overline\omega}(X^\ast,\fs_0)$ is one. In other words,
$$SW_{X^o,\xi_{\op{std}}}(\fs_0)=SW_{\RR^4,\xi_{\op{std}}}(\fs_0)=1$$
and so the relative Seiberg--Witten invariant cannot distinguish homotopy 4-spheres.
\end{thm}

\begin{proof}
Add a compactly supported perturbation to the $\overline\omega$-perturbed Seiberg--Witten equations~\eqref{SWafak} on $X^\ast$, namely that $\overline\omega$ shifts to $\overline\omega+\eta$, where
  \begin{equation*}
    \eta:=
    \begin{cases}
      -\overline\omega, \indent\indent\indent\,\text{on}\ X^o\cup\big([0,\frac\e2]\times S^3\big) \\
      0, \indent\indent\indent\indent\text{on}\ [\e,\infty)\times S^3\\
      \text{interpolate, on}\ [\frac\e2,\e]\times S^3
    \end{cases}
  \end{equation*}
Then the signed count of points of $\fM^{\op{afak}}_{\overline\omega}(X^\ast,\fs_0)$ is equal to that of the moduli space $\fM^{\op{afak}}_{\overline\omega+\eta}(X^\ast,\fs_0)$. The reason we consider this newly perturbed moduli space is to demonstrate a gluing formula below.

View $X^\ast$ diffeomorphically as a connected sum\footnote{As a reminder, the connected sum operation is well-defined by Theorem~\ref{CerfPalais}.}
$$X^\ast=X^o\cup_{S^3}(X^\ast-X^o)=X\#\RR^4$$
Then take a sequence of Riemannian metrics $\lbrace g^k\rbrace_{k\in\NN}$ on $X^\ast$ which ``pinch the neck'' along $\partial X^o=S^3$ as $k\to\infty$, and take a corresponding sequence of small perturbations $\lbrace\mu_\ast^k\rbrace_{k\in\NN}$ appearing in the $(\overline\omega+\eta)$-perturbed Seiberg--Witten equations~\eqref{SWafak} which vanish on a small neighborhood of $\partial X^o$ and are independent of $k$ on a slightly larger neighborhood of $\partial X^o$. The a priori estimates for SW solutions and a removable singularities theorem \cite{Salamon:removable} imply that a sequence of SW solutions on $(X^\ast,g^k)$ has a subsequence converging away from the neck to SW solutions over $X$ and $\RR^4$, yielding the gluing formula for sufficiently large $k$,
$$\fM^{\op{afak}}_{\overline\omega+\eta}(X^\ast,\fs_0)=\fM_{\mu_\ast^\infty}(X,\fs_0)\times\fM^{\op{afak}}_{\overline\omega'}(\RR^4,\fs_0)$$
where (by construction) the 2-form $\overline\omega'$ vanishes on a small ball around the origin in $\RR^4$ and agrees with $\omega_0$ on its AFAK end. 

It is not known whether $X$ admits metrics of positive scalar curvature, so we cannot rule out the existence of \textit{irreducible} unperturbed SW solutions in $\fM_0(X,\fs_0)$. But the virtual dimension of $\fM_{\mu_\ast^\infty}(X,\fs_0)$ is $d(\fs_0)=-1$, so the irreducible unperturbed SW solutions do not persist (under the small perturbation $\mu_\ast^\infty$) and hence $\fM_{\mu_\ast^\infty}(X,\fs_0)$ consists of a single point (which corresponds to the unique \textit{reducible} unperturbed SW solution). Thus,
$$\#\fM^{\op{afak}}_{\overline\omega+\eta}(X^\ast,\fs_0)=\#\fM^{\op{afak}}_{\overline\omega'}(\RR^4,\fs_0)$$
We are free to replace $\overline\omega'$ with the standard symplectic form $\overline\omega_{\op{std}}$ on $\RR^4$ (agreeing with $\omega_0$ on the AFAK end), by adding an appropriate compactly supported perturbation to the $\overline\omega'$-perturbed Seiberg--Witten equations~\eqref{SWafak} on $\RR^4$. The conclusions of the theorem now follow from the fact \cite{KM:contact}*{Theorem 4.2} that $\fM^{\op{afak}}_{\overline\omega_{\op{std}}}(\RR^4,\fs_0)$ consists of the unique finite-energy SW solution on $\RR^4$ asymptotic to the canonical configuration $[A_\xi,(1,0)]$ on $S^3$.
\end{proof}

\begin{remark}
The ``neck pinching'' argument in the proof of Theorem~\ref{thm:SW} also proves that the usual Seiberg--Witten invariants of a closed 4-manifold $M$ with $b^2_+(M)>1$ are equal to those of $X\#M$ for any homotopy 4-sphere $X$. This folklore result was certainly known to experts; see for example \cite{Vidussi:circle}. For completeness, it is known that the Bauer--Furuta invariants are not sensitive enough either \cite{BauerFuruta2}*{Proposition 2.3}.
\end{remark}


\subsection{Mixed ends}
\label{Mixed ends}

We will now relate the previous two subsections by a neck-stretching procedure along the submanifold $\partial X^o=\lbrace0\rbrace\times S^3\subset X^\ast$. To do this we must first introduce one more version of the Seiberg--Witten equations, this time on $\RR\times S^3$. We equip $\RR\times S^3$ with the metric $g=e^{2T}(ds^2+g_o)$ and the 2-form $\overline\omega=e^{2T}(ds\wedge\lambda_{\op{std}}+\frac12d\lambda_{\op{std}})$, where $T:\RR\to[0,\infty)$ is the smooth non-decreasing function that has value 0 on $(-\infty,0]$ and agrees with the function $T$ defined in Section~\ref{AFAK} on $[0,\infty)\times S^3$. So the metric is cylindrical on $(-\infty,\e]\times S^3$ and AFAK on $[1,\infty)\times S^3$. Finally, we fix a suitably generic exact 2-form $\mu\in\Omega^2(\RR\times S^3)$ that has exponential decay on $[\e,\infty)\times S^3$ and agrees on $(-\infty,\e]\times S^3$ with an exact 2-form $\mu\in\Omega^2(S^3)$ from Section~\ref{Contact 3-manifolds}.

Then a configuration $[\bA,\Psi]$ on $\RR\times S^3$ with respect to $\fs_0$, which is asymptotic on its positive end to the canonical configuration $[A_\xi,(1,0)]$ on $S^3$ (specified in Section~\ref{Contact 3-manifolds}) and asymptotic on its negative end to $\fc\in\fM(S^3,\fs_0)$, solves \textit{the mixed-perturbed Seiberg--Witten equations} when
\begin{equation}
\label{SWmixed}
D_\bA\Psi=0,\;\;F^+_\bA=r(\rho(\Psi)-i\overline\omega)+2i\mu_\ast
\end{equation}
for a fixed positive real number $r\in\RR$. The moduli space of such solutions is denoted by $\fM^{\op{mix}}(\fc)$. These equations were written down and analyzed in \cite{Taubes:Weinstein2}*{\S4}.\footnote{Here is where we need the errata to \cite{KM:contact}, provided in \cite{Zhang:monopoleFloer}*{\S5}. Specifically, details were missing in the proof of \cite{KM:contact}*{Lemma 3.17} that yields a uniform exponential decay estimate for the SW solutions on a compact 4-manifold with a positive AFAK end attached. The details were then supplied in a remark surrounding \cite{MrowkaRollin}*{Lemma 2.2.7}, which involves a bound by the volume of the compact 4-manifold. But if the compact 4-manifold is replaced by a 4-manifold with a negative cylindrical end (and a positive AFAK end attached), the details are instead supplied by \cite{Zhang:monopoleFloer}*{Theorem 5.13}.} In particular,

\begin{prop}[\cite{Taubes:Weinstein2}*{Proposition 4.3}]
\label{prop:Weinstein}
For $r$ sufficiently large and for the monopole $\fc_\varnothing$ that corresponds to the empty orbit set of $(S^3,\lambda_{\op{std}})$ in Theorem~\ref{thm:generators}, $\fM^{\op{mix}}(\fc_\varnothing)$ has dimension zero and consists of a single point.
\end{prop}

The main theorem of this section is the following.

\begin{thm}
\label{thm:mix}
In the setting of Theorem~\ref{thm:SW}, $SW_{X^o,\xi_{\op{std}}}(\fs_0)$ is equal to
$$\#\fM^{\op{afak}}_{\overline\omega}(X^\ast,\fs_0)=\#\fM^{\op{cyl}}_{\hat\omega}(X^\ast,\fc_\varnothing;\fs_0)$$
for the versions of~\eqref{SWcyl} and~\eqref{SWafak} with $r$ sufficiently large.
\end{thm}

\begin{proof}
We may follow the analogous proof of the composition law in monopole Floer homology \cite{KM:book}*{\S26}. That is, we ``stretch the neck'' of $X^\ast$ (along $S^3=\partial X^o$) and use the large $r$ $\lambda_{\op{std}}$-perturbations of~\eqref{SW3} on the cylindrical neck region. Then since $\dim\fM^{\op{afak}}_{\overline\omega}(X^\ast,\fs_0)=0$ we use \cite{KM:book}*{Proposition 26.1.6} to compute
$$SW_{X^o,\xi_{\op{std}}}(\fs_0)=\sum_\fc\#\fM^{\op{cyl}}_{\hat\omega}(X^\ast,\fc;\fs_0)\cdot\#\fM^{\op{mix}}(\fc)$$
where the index is over those monopoles $\fc\in\fM(S^3,\fs_0)$ for which
\begin{equation}
\label{stretchdim}
\dim\fM^{\op{cyl}}_{\hat\omega}(X^\ast,\fc;\fs_0)=\dim\fM^{\op{mix}}(\fc)=0
\end{equation}
So it suffices to show that the only monopole satisfying~\eqref{stretchdim} is $\fc_\varnothing$, as we may then invoke Proposition~\ref{prop:Weinstein} to conclude the proof of the theorem. By additivity of gradings \cite{KM:book}*{\S24.4} and the fact that negative dimensional moduli spaces of irreducible SW solutions are empty when cut out transversally, if~\eqref{stretchdim} holds for $\fc_\varnothing$ then it holds for no others. For $\fc_\varnothing$ we invoke Proposition~\ref{prop:Weinstein} and $\dim\fM^{\op{cyl}}_{\hat\omega}(X^\ast,\fc_\varnothing;\fs_0)=\dim\fM^{\op{afak}}_{\overline\omega}(X^\ast,\fs_0)=0$ via \cite{KM:contact}*{\S5.4}.
\end{proof}

\section{Final result and outlook}
\label{Final result and outlook}

The following proposition shows that $Gr_{X,\omega}$ equivalently counts SW solutions on a completion of $X^o-\nN$ and may be viewed as an element of $\bigotimes_{k=1}^N\Hfrom^{-\ast}(S^1\times S^2,\fs_{\xi_{\op{ns}}}+1)$. In accord with \cite{Gerig:taming, Gerig:Gromov} we normalize the $\ZZ$-grading on each Floer homology factor so that
  \begin{equation*}
    ECH_j(S^1\times S^2,\xi_{\op{ns}},1)\cong\Hfrom^j(S^1\times S^2,\fs_{\xi_{\op{ns}}}+1)\cong
    \begin{cases}
      \ZZ/2, & \text{if}\ j\ge0 \\
      0, & \text{otherwise}
    \end{cases}
  \end{equation*}
The theorem proceeding the proposition then computes $Gr_{X,\omega}$.

\begin{prop}
\label{prop:multivalued}
For generic $J$, sufficiently large $r$, and an admissible orbit set $\Theta$ on $(\partial\nN,\xi_{\op{ns}})$ with action less than $\rho(A_\1)$, there is a multi-valued bijection between $\mM_0(\varnothing,\Theta;A_\1)$ and $\fM_0(\fc_\varnothing,X^o-\nN,\fc_\Theta;\fs_0)$ such that
$$\#\mM_0(\varnothing,\Theta;A_\1)=\#\fM_0(\fc_\varnothing,X^o-\nN,\fc_\Theta;\fs_0)\in\ZZ/2$$
\end{prop}

\begin{proof}
For the same reason as in the proof of Theorem~\ref{thm:class}, we may follow \cite{Gerig:Gromov}*{\S4} verbatim. We recall that the correspondence $\Theta\mapsto\fc_\Theta$ is given by Theorem~\ref{thm:generators}. \end{proof}

\begin{thm}
\label{thm:Gr}
Given a homotopy 4-sphere $X$ and asymptotically standard near-symplectic form $\omega$ on $X^\ast$,
$$Gr_{X,\omega}=SW_{X^o,\xi_{\op{std}}}(\fs_0)\in\ZZ/2$$
and so the near-symplectic Gromov invariant $Gr_{X,\omega}$ cannot distinguish homotopy 4-spheres.
\end{thm}

\begin{proof}
We may follow \cite{Gerig:Gromov}*{\S5} verbatim, using the moduli space $\fM^{\op{cyl}}_{\hat\omega}(X^\ast,\fc_\varnothing;\fs_0)$ in replace of the moduli spaces which contribute to the Seiberg--Witten invariants of a closed 4-manifold. Specifically, with respect to the decomposition $X^o=(X^o-\nN)\cup\nN$, we ``stretch the neck'' along $\partial\nN$ and analyze what happens to $\fM^{\op{cyl}}_{\hat\omega}(X^\ast,\fc_\varnothing;\fs_0)$. Then we invoke Proposition~\ref{prop:multivalued} to conclude that
$$Gr_{X,\omega}=\#\fM^{\op{cyl}}_{\hat\omega}(X^\ast,\fc_\varnothing;\fs_0)\in\bigotimes_{k=1}^N ECH_0(S^1\times S^2,\xi_{\op{ns}},1)\cong\bigotimes_{k=1}^N\Hfrom^0(S^1\times S^2,\fs_{\xi_{\op{ns}}}+1)\cong\ZZ/2$$
where we blur the distinction between homology class and numerical invariant by pairing the homology class with the generator of the Floer group. That is, $Gr_{X,\omega}$ lives in the absolute grading of Floer homology for which it may be identified with an integer modulo 2, and this number is equal to $SW_{X^o,\xi_{\op{std}}}(\fs_0)$ by Theorem~\ref{thm:SW} and Theorem~\ref{thm:mix}.
\end{proof}

\subsection{Related idea}
\label{Related idea} Take $N=2$ for simplicity of discussion. The invariant $Gr_{X,\omega}$ only counted pseudoholomorphic curves in $(\overline{X^o-\nN},\omega)$ which had no positive ends, and so it may be viewed as the image of the generator $[\varnothing]\in ECH_0(S^3,\xi_{\op{std}},0)\cong\ZZ/2$ under the ECH cobordism map
$$\Phi_0:ECH_0(S^3,\xi_{\op{std}},0)\to ECH_0(S^1\times S^2,\xi_{\op{ns}},1)\otimes ECH_0(S^1\times S^2,\xi_{\op{ns}},1)$$
Since $ECH_{2k}(S^3,\xi_\text{std},0)\cong\ZZ/2$ for all $k\ge0$ and vanishes otherwise, we may define similar ECH (or monopole Floer) cobordism maps by counting index 0 pseudoholomorphic curves (or SW solutions) with certain positive and negative ends,
$$\Phi_{2k}:ECH_{2k}(S^3,\xi_\text{std},0)\to \bigoplus_{i+j=2k}ECH_i(S^1\times S^2,\xi_{\op{ns}},1)\otimes ECH_j(S^1\times S^2,\xi_{\op{ns}},1)$$
Take $k=1$ for example, so that $\Phi_2(\text{generator})\in(\ZZ/2)^3$ in the only nontrivial gradings $(i,j)\in\lbrace(0,2),\;(1,1),\;(2,0)\rbrace$. We question whether the values in these gradings contain more information than the invariant $Gr_{X,\omega}=\Phi_0(\text{generator})\in\ZZ/2$ in grading $(0,0)$. Unfortunately, they do not and we will explain the following:
$$\Phi_2(\text{generator})=\left(Gr_{X,\omega},0,Gr_{X,\omega}\right)\in(\ZZ/2)^3$$

The $(0,2)$ and $(2,0)$ gradings reduce to the $(0,0)$ grading thanks to the \textit{U-maps}. The U-maps are degree $-2$ maps $U:ECH_j(S^1\times S^2,\xi_{\op{ns}},1)\to ECH_{j-2}(S^1\times S^2,\xi_{\op{ns}},1)$ and $U:ECH_j(S^3,\xi_\text{std},0)\to ECH_{j-2}(S^3,\xi_\text{std},0)$, which are isomorphisms for $j\ge2$. We can compose the U-maps on either side of the cobordism
$$\Phi_{2k-2}\circ U=(U\otimes\1)\circ\Phi_{2k}=(\1\otimes U)\circ\Phi_{2k}$$
as explained in \cite{KM:book}*{\S3.4} and \cite{Hutchings:lectures}*{\S3.8}.

The $(1,1)$ grading vanishes thanks to the \textit{loop-maps}. Each loop-map is a degree $-1$ map $\triangle_\gamma:ECH_j(S^1\times S^2,\xi_{\op{ns}},1)\to ECH_{j-1}(S^1\times S^2,\xi_{\op{ns}},1)$ defined using a generator $\gamma\in H_1(S^1\times S^2;\ZZ)$, which is an isomorphism for odd $j\ge1$ (and satisfies $\triangle_\gamma\circ\triangle_\gamma=0$). We can compose the loop-maps, as explained in \cite{KM:book}*{\S3.4} and \cite{Hutchings:lectures}*{\S3.8}, to obtain
$$(\triangle_\gamma\otimes\triangle_{\gamma'})\circ\Phi_{2k}=0$$
because the generators $(\gamma,\gamma')\in H_1(S^1\times S^2;\ZZ)\oplus H_1(S^1\times S^2;\ZZ)$ become homologous in the cobordism $X^o-\nN$.

\subsection{Symplectic field theory}
\label{Symplectic field theory}

We could not mimic the construction of $Gr_{X,\omega}$ using SFT because the contact homology of any overtwisted contact 3-manifold is trivial \cite{Yau:overtwisted, BourgeoisNiederkruger}. That is, a tentative SFT-type invariant would use moduli spaces of curves in the SFT framework and subsequently represent an element of a contact homology $CH_\ast(\bigsqcup_{k=1}^NS^1\times S^2,\xi_{\op{ns}})=0$.

\appendix
\section{Proof of Theorem~\ref{LuttingerTaubes}}
\label{appendix}

The proof is an application of $L^2$ Hodge theory on manifolds with cylindrical ends. Equip $X^\ast$ with a metric $g$ so that it may be written as the cylindrical completion of a Riemannian homotopy 4-ball $(M,g)$ with boundary $(S^3,g_\text{round})$. Denote by $\lambda$ the contact form $\frac12(x_1dy_1-y_1dx_1+x_2dx_2-y_2 dx_2)$ on $S^3$, and arbitrarily extend the self-dual 2-form $d(e^s\lambda)$ on $[0,\infty)\times S^3$ to an exact 2-form $\eta$ on $X^\ast$.

Define $\eta^a:=\eta+da$ for Sobolev 1-forms $a$ on $X^\ast$. We first aim to solve $d^\ast\eta^a=0$ (note that $d\eta^a=0$ automatically), or equivalently $d^\ast da=-d^\ast\eta$. It suffices to solve
$$(d^\ast d+dd^\ast)a=-d^\ast\eta$$
because co-exact forms are orthogonal to exact forms (so then $dd^\ast a=0$). This Laplace equation has a smooth $L^2_2$ solution $a_\eta$ on $X^\ast$ by $L^2$ Hodge theory \cite{Lockhart:FredholmHodge}, noting that $d^\ast\eta$ is smooth and $d^\ast\eta|_{[0,\infty)\times S^3}=d\eta|_{[0,\infty)\times S^3}=0$.

Thus $\omega:=(\eta^{a_\eta})_+$ is a self-dual 2-form on $(X^\ast,g)$ with $\omega|_{\lbrace s\rbrace\times S^3}=d(e^s\lambda)+(da)_+|_{\lbrace s\rbrace\times S^3}$ and it is closed (hence harmonic) because
$$0=d\eta^{a_\eta}=d\omega+d(\eta^{a_\eta})_-$$
$$0=d^\ast\eta^{a_\eta}=d\omega-d(\eta^{a_\eta})_-$$
Note that $(da)_+\in H^2([0,\infty)\times S^3;\RR)\cong 0$ because $0=d\omega=dd(e^s\lambda)$ there, so we may write $(da)_+=d\alpha$ on $[0,\infty)\times S^3$ for some $\alpha\in\Omega^1(S^3;\RR)$. Let's clarify nondegeneracy of $\omega$. Since $a$ satisfies an elliptic equation with an $L^2$ bound, the $L^2$ norm of $a$ decreases to zero on any sequence of balls going off to infinity. By elliptic bootstrapping we get small $C^k$ bounds on $a$, hence small pointwise bounds for $\alpha$ and $d\alpha$ on $[s_0,\infty)\times S^3$ for $s_0\in\RR$ sufficiently large. This guarantees $\omega\wedge\omega>0$ on $[s_0,\infty)\times S^3$.

Now by the work of Honda \cite{Honda:transversality}*{\S2.7} adapted to this relative setting, we find a near-symplectic self-dual harmonic form $\omega'$ on $X^\ast$ with respect to some metric $g'$, such that $(\omega',g')$ agrees with $(\omega,g)$ on $[s_0,\infty)\times S^3$. We then build the 2-form
$$\omega'':=\omega'-d(\rho(s)\alpha)$$
with $\rho(s)=0$ for $s\le s_0$ and $\rho(s)=1$ for $s\gg s_0$ such that $\frac{d\rho}{ds}$ is sufficiently small. Then $\omega''$ is a near-symplectic form on $X^\ast$ (though no longer self-dual with respect to $g'$) which agrees with the standard symplectic form on $\RR^4$ outside some compact set.


\end{document}